\documentclass[dvips,sts]{imsart}

\usepackage{amsthm,amsmath, amssymb}
\RequirePackage[colorlinks,citecolor=blue,urlcolor=blue]{hyperref}

\RequirePackage{hypernat}

\usepackage{amsmath, moreverb, hyperref}
\hypersetup{colorlinks=true,citecolor=blue}

\usepackage{color}
\usepackage{amsfonts, fancybox}
\usepackage{amssymb}
\usepackage[american]{babel}
\usepackage{graphicx, pstricks, pst-node, color}
\usepackage{psfrag}\usepackage{subfigure}
\usepackage{flafter, bm, dsfont}
\usepackage[section]{placeins}

\startlocaldefs

\newcommand{\cA}{\mathcal{A}}
\newcommand{\cB}{\mathcal{B}}

\newcommand{\cD}{\mathcal{D}}

\newcommand{\cG}{\mathcal{G}}
\newcommand{\cH}{\mathcal{H}}

\newcommand{\cN}{\mathcal{N}}

\newcommand{\cS}{\mathcal{S}}
\newcommand{\cT}{\mathcal{T}}

\newcommand{\bE}{\mathbf{E}}

\newcommand{\bP}{\mathbf{P}}

\newcommand{\del}{\delta}
\newcommand{\eps}{\varepsilon}

\newcommand{\SDP}{\mathsf{\mathop{SDP}}}

\newcommand{\tr}{\mathbf{Tr}}
\newcommand{\rk}{\mathbf{rank}}

\newcommand{\bp}{\mathbf{p}}

\newcommand{\Sh}{\hat{\Sigma}}
\newcommand{\Dh}{\hat{\Delta}}
\newcommand{\Ph}{\hat{\Psi}}
\newcommand{\N}{\mathbb{N}}
\newcommand{\Sy}{\mathbf{S}}
\newcommand{\R}{\mathbf{R}}
\newcommand{\Pro}{\mathbf{P}\kern-0.12em}
\newcommand{\lkm}{\lambda_{\max}^k}
\newcommand{\D}{\mathcal{D}}

\newcommand{\fP}{\mathbf{P}}

\newcommand{\E}{\mathbb{E}}

\newcommand{\bl}{{\sf bl}}
\newcommand{\var}{{\sf V}}
\newcommand{\hvar}{{\sf\widehat{V}}_n}

\newcommand{\ubar}[1]{\underline{#1}}

\newcommand{\G}{\mathbb{G}}

\newenvironment{hypothesis}[1]{\begin{trivlist}
\item[\hskip \labelsep {\bfseries Hypothesis #1}]}{\end{trivlist}}

\makeatletter
\def\namedlabel#1#2{\begingroup
   \protect\def\@currentlabel{#2}%
   \label{#1}\endgroup
}
\makeatother

\newcommand{\BlackBox}{\rule{1.5ex}{1.5ex}}
\renewenvironment{proof}{\par\noindent{\bfseries\upshape
  Proof\ }}{\hfill\BlackBox\\[2mm]}

\newtheorem{theorem}{Theorem}
\newtheorem{lemma}[theorem]{Lemma}

\newtheorem{corollary}[theorem]{Corollary}
\newtheorem{definition}[theorem]{Definition}

\begin{document}

\begin{frontmatter}

\title{Computational Lower Bounds for Sparse PCA}
\runtitle{Computational Lower Bounds for Sparse PCA}

 \author{\fnms{Quentin} \snm{Berthet}\thanksref{t1, t2}\ead[label=e1]{qberthet@princeton.edu}\and \fnms{Philippe} \snm{Rigollet}\corref{}\thanksref{t2}\ead[label=e2]{rigollet@princeton.edu}}

\thankstext{t1}{Partially supported by a Gordon S. Wu Fellowship.}
\thankstext{t2}{Partially supported by NSF grants DMS-0906424, CAREER-DMS-1053987.}

 \affiliation{Princeton University}

 \address{{Quentin Berthet}\\
{Department of Operations Research} \\
{ and Financial Engineering}\\
{Princeton University}\\
{Princeton, NJ 08544, USA}\\
 \printead{e1}}
 \address{{Philippe Rigollet}\\
{Department of Operations Research} \\
{ and Financial Engineering}\\
{Princeton University}\\
{Princeton, NJ 08544, USA}\\
 \printead{e2}}

%

\runauthor{Berthet and Rigollet}

\begin{abstract}
\ In the context of sparse principal component detection, we bring evidence towards the existence of a statistical price to pay for computational efficiency. We measure the performance of a test by the smallest signal strength that it can detect and we propose a computationally efficient method based on semidefinite programming. We also prove that the statistical performance of this test cannot be strictly improved by any computationally efficient method. Our results can be viewed as complexity theoretic lower bounds conditionally on the assumptions that some instances of the planted clique problem cannot be solved in randomized polynomial time.
\end{abstract}



\end{frontmatter}

\section{Introduction}

The modern scientific landscape has been significantly transformed over the past decade by the apparition of massive datasets.  From the statistical learning point of view, this transformation has led to a paradigm shift. Indeed, most novel methods consist in \emph{searching} for sparse structure in datasets, whereas \emph{estimating} parameters over this structure is now a fairly well understood problem. It turns out that most interesting structures have a combinatorial nature, often leading to computationally hard problems. This has led researchers to consider various numerical tricks, chiefly convex relaxations, to overcome this issue. While these new questions have led to fascinating interactions between learning and optimization, they do not always come with satisfactory answers from a statistical point of view. The main purpose of this paper is to study one example, namely sparse principal component detection, for which current notions of statistical optimality should also be shifted, along with the paradigm.


Sparse detection problems where one wants to detect the presence of a sparse structure in noisy data falls in this line of work. There has been recent interest in detection problems of the form signal-plus-noise, where the signal is a vector with combinatorial structure \cite{AddBroDev09, AriCanPla11, AriVer13} or even a matrix \cite{ButIng13,SunNob13, KolBalRin11,BalKolRin11}.  The matrix detection problem was pushed beyond the signal-plus-noise model towards more complicated dependence structures  \cite{AriBubLug12, AriBubLug13, BerRig12}. One contribution of this paper is to extend these results to more general distributions.

For matrix problems, and in particular sparse principal component (PC) detection, some computationally efficient methods have been proposed, but they are not proven to achieve the optimal detection levels. \cite{JohLu09, CaiMa12,Ma13} suggest heuristics for which detection levels are unknown and \cite{BerRig12} prove suboptimal detection levels for a natural semidefinite relaxation developed in \cite{dAsGhaJor07} and an even simpler, efficient, dual method called Minimum Dual Perturbation (MDP). More recently, \cite{dAsBacGha12} developed another semidefinite relaxation for sparse PC detection that performs well only outside of the high-dimensional, low sparsity regime that we are interested in. Note that it follows from the results of \cite{AmiWai08} that the former semidefinite relaxation is optimal if it has a rank-one solution. Unfortunately, rank-one solutions can only be guaranteed at suboptimal detection levels. This literature hints at a potential cost for computational efficiency in the sparse PC detection problem.

Partial results were obtained in~\cite{BerRig12} who proved that their bound for MDP and SDP are unlikely to be improved, as otherwise they would lead to randomized polynomial time algorithms for instances of the planted clique problem that are believed to be hard. This result only focuses on a given testing method, but suggests the existence of an intrinsic gap between the optimal rates of detection and what is statistically achievable in polynomial time. Such phenomena are hinted at in \cite{ChaJor13} but their these results focus on the behavior of upper bounds. Closer to our goal, is  \cite{ShaShaTom12} that exhibits a statistical price to pay for computational efficiency. In particular, their derive a computational theoretic lower bound using a much weaker conjecture than the hidden clique conjecture that we employ here, namely the existence of one-way permutations. This conjecture is widely accepted and is the basis of many cryptographic protocols. Unfortunately, the lower bound holds only for a synthetic classification problem that is somewhat tailored to this conjecture.  It still remains to fully describe a theory, and to develop lower bounds on the statistical accuracy that is achievable in reasonable computational time for natural problems. This article aims to do so for a general sparse PC detection problem. 

This paper is organized in the following way. The sparse PC detection problem is formally described in Section~\ref{SEC:pdesc}. Then, we show in Section~\ref{SEC:optesting} that our general detection framework is a natural extension of the existing literature, and that all the usual results for classical detection of sparse PC are still valid. Section~\ref{SEC:ptesting} focuses on testing in polynomial time, where we study detection levels for the semidefinite relaxation developed of \cite{dAsGhaJor07}  (It trivially extends to the MDP statistic of \cite{BerRig12}). These levels are shown to be unimprovable using computationally efficient methods  in Section~\ref{SEC:CTLB}. This is achieved by introducing a new notion of optimality that takes into account computational efficiency. Practically, we reduce the planted clique problem, conjectured to be computationally hard already in an~\emph{average-case} sense (i.e. over most random instances) to obtaining better rates for sparse PC detection. 

{\sc Notation}.
The space of $d \times d$ symmetric real matrices is denoted by $\Sy_d$. We write $Z \succeq 0$ whenever $Z$ is semidefinite positive. We denote by $\N$ the set of nonnegative integers and define $\N_1=\N\setminus\{0\}$.

The elements of a vector $v \in \R^d$ are denoted by $v_1, \ldots, v_d$ and similarly, a matrix $Z$ has element $Z_{ij}$ on its $i$th row and $j$th column. For any $q > 0$,   $|v|_q$ denotes the $\ell_q$ ``norm" of a vector $v$ and is defined by $|v|_q=(\sum_{j}|v_j|^q)^{1/q}$. Moreover, we denote by $|v|_0$ its so-called $\ell_0$ ``norm", that is its number of nonzero elements.  Furthermore, by extension, for $Z \in \Sy_d$, we denote by $|Z|_q$ the $\ell_q$ norm of the vector formed by the entries of $Z$. We also define for $q\in [0,2)$ the set $\cB_q(R)$ of unit vectors within the $\ell_q$-ball of radius $R>0$
$$\cB_q(R) = \{v \in \R^d \,:\, |v|_2=1\,, |v|_q \le R\}\, .$$

For a finite set $S$, we denote by $|S|$ its cardinality. We also write $A_S$ for the $|S|\times|S|$ submatrix with elements $(A_{ij})_{i,j \in S}$, and $v_S$ for the vector of $\R^{|S|}$ with elements $v_i$ for $i \in S$. The vector $\mathbf{1}$ denotes a vector with coordinates all equal to $1$. If a vector has an index such as $v_i$, then we use $v_{i,j}$ to denote its $j$th element. 

The vectors $e_i$ and matrices $E_{ij}$ are the elements of the canonical bases of $\R^d$ and $\R^{d\times d}$. 
We also define $\cS^{d-1}$ as the unit Euclidean sphere of $\R^d$ and $\cS^{d-1}_S$ the set of vectors in $\cS^{d-1}$ with support $S \subset\{1, \ldots, d\}$. The identity matrix in $\R^d$ is denoted by~$I_d$.

A Bernoulli random variable with parameter $p \in [0,1]$ takes values $1$ or $0$ with probability $p$ and $1-p$ respectively. A Rademacher random variable takes values $1$ or $-1$ with probability $1/2$. A binomial random variable, with distribution $\cB(n,p)$ is the sum of $n$ independent Bernoulli random variables with identical parameter $p$. A hypergeometric random variable, with distribution $\cH(N,k,n)$ is the random number of successes in $n$ draws from a population of size $N$ among which are $k$ successes, without replacement. The total variation norm, noted $\| \cdot \|_{\sf TV}$ has the usual definition.

The trace and rank functionals are denoted by $\tr$ and $\rk$ respectively and have their usual definition.  We denote by $T^c$ the complement of a set $T$. Finally, for two real numbers $a$ and $b$, we write $a \wedge b=\min(a,b)$, $a \vee b = \max(a,b)$, and $a_+ = a \vee 0 \,$.

\section{Problem description}
\label{SEC:pdesc}
Let $X \in \R^d$ be a centered random vector with unknown distribution $\fP$ that has finite second moment along every direction. The first principal component for $X$ is a direction $v\in  \cS^{d-1}$ such that the variance $\var(v)=\bE[(v^\top X)^2]$ along direction $v$ is larger than in any other direction. If no such $v$ exists, the distribution of $X$ is said to be \emph{isotropic}. The goal of sparse principal component detection is to test whether $X$ follows an isotropic distribution $\fP_0$ or a distribution $\fP_v$ for which there exists a sparse $v \in \cB_0(k)$, $k \ll d$, along which the variance is large.  Without loss of generality, we assume that under the isotropic distribution $\fP_0$, all directions have unit variance and under $\fP_v$, the variance along $v$ is equal to $1+\theta$ for some positive $\theta$. Note that since $v$ has unit norm, $\theta$ captures the signal strength. 

To perform our test, we observe $n$ independent copies $X_1,\ldots,X_n$ of $X$. For any direction $u\in \cS^{d-1}$, define the empirical variance along $u$ by
$$
\hvar(u)=\frac1n\sum_{i=1}^n \big(u^\top X_i \big)^2\,.
$$
Clearly the concentration of $\hvar(u)$ around $\var(u)$ will have a significant effect on the performance of our testing procedure. If, for any $u \in \cS^{d-1}$, the centered random variable $(u^\top X)^2-\E[(u^\top X)^2]$ satisfies the conditions for Bernstein's inequality (see, e.g., \cite{Mas07}, eq. (2.18), p.24) under both $\fP_0$ and $\fP_v$, then, up to numerical constants, we have
\begin{align}
\sup_{u \in \cS^{d-1}}\fP_0^{\otimes n}\Big(\big|\hvar(u)-1\big|>  4\sqrt{\frac{\log(1/\nu)}{n}} + 4 \frac{\log(1/\nu)}{n} \Big) &\le \nu\,,  &\quad\forall \nu>0\,,&\label{EQ:H0hpb}\\
\fP_v^{\otimes n} \Big(\hvar(v)-(1+\theta) < - 2 \sqrt{\frac{2\theta k \log(2/\nu)}{n}} - 4 \frac{\log(2/\nu)}{n} \Big)  &\le \nu\,, &\quad \forall \nu>0, \ v \in \cB_0(k)&\,.\label{EQ:H1hpb}
\end{align}
Such inequalities are satisfied  if we assume that $\fP_0$ and $\fP_v$ are sub-Gaussian distributions for example.
%
%
%
Rather than specifying such an ad-hoc assumption, we define the following sets of distributions under which the fluctuations of $\hvar$ around $\var$ are of the same order as those of sub-Gaussian distributions. As a result, we formulate our testing problem on the unknown distribution $\fP$ of $X$ as follows
\begin{eqnarray*}
H_0&:& \fP\in \cD_0 = \big\{ \fP_0 \,:\,  \eqref{EQ:H0hpb}  \ \text{holds} \big\} \\
H_1&:& \fP\in\cD_1^k(\theta) = \bigcup_{v \in \cB_0(k)}\big\{ \fP_v \,:\,   \eqref{EQ:H1hpb}\ \text{holds} \big\}\, .
\end{eqnarray*}
Note that distributions in $ \cD_0$ and $\cD_1^k(\theta)$ are implicitly centered at zero.

We argue that interesting testing procedures should be robust and thus perform well uniformly over these distributions. In the rest of the paper, we focus on such procedures. The existing literature on sparse principal component testing, particularly in \cite{BerRig12} and \cite{AriBubLug12} focuses on multivariate normal distributions, yet only relies on the sub-Gaussian properties of the empirical variance along unit directions. Actually, all the distributional assumptions made in \cite{VuLei12,AriBubLug12} and \cite{BerRig12} are particular cases of these hypotheses. We will show that concentration of the empirical variance as in~\eqref{EQ:H0hpb} and \eqref{EQ:H1hpb} is sufficient to derive the results that were obtained under the sub-Gaussian assumption. 

Recall that a test for this problem is a family $\psi=\{\psi_{d,n,k}\}$ of $\{0,1\}$-valued measurable functions of the data $(X_1, \ldots, X_n)$. Our goal is to quantify the smallest signal strength $\theta>0$ for which there exists a test $\psi$ with maximum test error bounded by $\delta>0$, i.e.,
$$
\sup_{\begin{subarray}{l}\fP_0 \in \cD_0\\\fP_1 \in \cD_1^k (\theta)\end{subarray}} \Big\{ \fP_0^{\otimes n}(\psi=1) \vee \fP_1^{\otimes n}(\psi=0) \Big \} \le \del \, .
$$
To call our problem ``sparse", we need to assume somehow that $k$ is rather small. Throughout the paper, we fix a tolerance $0<\delta<1/3$ (e.g., $\delta=5\%$) and focus on the case where the parameters are in the \emph{sparse regime} $R_0 \subset \N_1^3$ of positive integers defined by
$$
R_0=\Big\{(d,n,k) \in \N_1^3\,:\, 15\sqrt{\frac{k\log (6ed/\delta)}{n}}\le1\,, k \le d^{0.49}\}\,.
$$
Note that the constant $0.49$ is arbitrary and can be replaced by any constant $C<0.5$.\begin{definition}
Fix a set of parameters $R \subset R_0$ in the sparse regime. Let $\cT$ be a set of tests.  A function $\theta^*$ of $(d,n,k) \in R$ is called {\bf optimal rate of detection over the class}~$\cT$ if for any $(d,n,k)\in R$, it holds:
\begin{itemize}
\item[(i)] there exists a test $\psi \in \cT$ that  discriminates between $H_0$ and $H_1$  at level $\bar c\theta^*$ for some constant $\bar c>0$, i.e., for any $\theta\ge \bar c \theta^*$
$$
\sup_{\begin{subarray}{l}\fP_0 \in \cD_0\\\fP_1 \in \cD_1^k (\theta)\end{subarray}} \Big\{ \fP_0^{\otimes n}(\psi=1) \vee \fP_1^{\otimes n}(\psi=0) \Big \} \le \del \,.
$$
In this case we say that $\psi \in \cT$  discriminates between $H_0$ and $H_1$  at rate~$\theta^*$.
\item[(ii)] for any test $\phi \in \cT$, there exists a constant $\ubar{c}_\phi>0$ such that $\theta\le  \ubar{c}_\phi \theta^*$ implies
$$
\sup_{\begin{subarray}{l}\fP_0 \in \cD_0\\\fP_1 \in \cD_1^k (\theta)\end{subarray}}\Big\{ \fP_0^{\otimes n}(\phi=1) \vee \fP_1^{\otimes n}(\phi=0) \Big \} \ge \delta\,.
$$
\end{itemize}
Moreover, if both~(i) and~(ii) hold, we say that $\psi$ is an optimal test over the class $\cT$.
%
\end{definition}
This an adaptation of the usual notion of statistical optimality, when one is focusing on the class of measurable functions, for $\psi_{d,n,k}:(X_1, \ldots, X_n) \mapsto \{0,1\}$, also known as minimax optimality \cite{Tsy09}. In order to take into account the asymptotic nature of some classes of statistical tests (namely, those that are computationally efficient), we allow the constant $\ubar{c}_\phi$ in {\it (ii)} to depend on the test.
\section{Statistically optimal testing}
\label{SEC:optesting}
We focus first on the traditional setting where   $\cT$ contains all sequences $\{\psi_{d,n,k}\}$ of tests.

Denote by $\Sigma=\E[XX^\top]$ the covariance matrix of $X$ and by $\Sh$ its empirical counterpart:
\begin{equation}
\label{EQN:sh}
\Sh= \frac 1n \sum_{i=1}^n X_i X_i^\top \, .
\end{equation}
Observe that   $\var(u)=u^\top \Sigma \,u$ and $\hvar(u)=u^\top \Sh u$, for any $u \in \cS^{d-1}$. Maximizing $\hvar(u)$ over $\cB_0(k)$ gives the largest empirical variance along any $k$-sparse direction. It is also known as the $k$-sparse eigenvalue of $\Sh$ defined by
\begin{equation}
\label{EQ:deflkm}
\lkm(\Sh) =  \max_{u \in \cB_0(k)} u^\top \Sh\, u\, .
\end{equation}
%
%
The following theorem describes the performance of the test
\begin{equation}
\label{EQ:defpsi}
\psi_{d,n,k} = \mathbf{1}\{\lkm(\Sh)>1+\tau\}\,, \quad \tau >0\, .
\end{equation}
\begin{theorem}
\label{TH:uplkm}
Assume that $(d,n,k) \in R_0$ and define
$$
\bar\theta= 15 \sqrt{\frac{k \log\big(\frac{6ed}{k\delta}\big)}{n}}\,.
$$ 
Then, for $\bar\theta<\theta < 1$, the test $\psi$ defined in~\eqref{EQ:defpsi} with threshold $
\tau=8 \sqrt{\frac{k \log\big(\frac{6ed}{k\delta}\big)}{n}}$ ,
satisfies
$$
\sup_{\begin{subarray}{l}\fP_0 \in \cD_0\\\fP_1 \in \cD_1^k (\theta)\end{subarray}} \Big\{ \fP_0^{\otimes n}(\psi=1) \vee \fP_1^{\otimes n}(\psi=0) \Big \} \le \del \, .
$$
\end{theorem}
\begin{proof}
Define $\tau_1 = 7 \sqrt{ k \log(2/\delta)/n}$.
For $\fP_1 \in \cD^k_1(\theta)$, by~\eqref{EQ:H1hpb}, and for $\fP_0 \in \cD_0$, using  Lemma~\ref{LEM:lkmH0}, we get
$$
\fP_0^{\otimes n}\Big(\lkm(\Sh) \ge 1+ \tau \Big) \le \delta\,, \quad \fP_1^{\otimes n}\Big(\lkm(\Sh) \le 1+ \theta - \tau_1\Big) \le \delta \,.
$$
To conclude the proof, observe that  $\tau\le \bar \theta-\tau_1<\theta-\tau_1$.
%
%
\end{proof}
The following lower bound follows directly from \cite{BerRig12}, Theorem~5.1 and holds already for Gaussian distributions.
\begin{theorem}
\label{TH:infminimax}
For all $\eps>0$, there exists  a constant $C_\eps>0$  such that if
$$
\theta < \underline{\theta}_\eps=\sqrt{\frac{ k\log \left(C_\eps d/k^2+ 1\right)}{n}} \,,
$$
any test $\phi$ satisfies
\begin{equation*}
\sup_{\begin{subarray}{l}\fP_0 \in \cD_0\\\fP_1 \in \cD_1^k (\theta)\end{subarray}} \big\{ \fP_{0}^{\otimes n}(\phi=1) \vee \fP_1^{\otimes n}(\phi=0) \big\} \ge \frac 12 - \eps \, .
\end{equation*}
\end{theorem}
Theorems~\ref{TH:uplkm} and~\ref{TH:infminimax} imply the following result.
\begin{corollary}
\label{COR:minimax}
The sequence  
$$
\theta^*=\sqrt{\frac{k\log d}{n}}\,,\quad  (d,n,k) \in R_0\, ,
$$
is the optimal rate of detection over the class of \emph{all tests}.
\end{corollary}
\section{Polynomial time testing}
\label{SEC:ptesting}
It is not hard to prove that approximating $\lkm(A)$ up to a factor of $m^{1-\eps}, \eps>0$, \emph{for any} symmetric matrix $A$ of size $m \times m$ and any $k\in \{1, \ldots, m\}$ is  {\sf NP}-hard, by a trivial reduction to {\sf CLIQUE}  (see \cite{Has96,Has99,Zuc06} for hardness of approximation of {\sf CLIQUE}). Yet, our problem is not worst case and we need not consider \emph{any} matrix $A$. Rather, here, $A$ is a random matrix and we cannot directly apply the above results. 

In this section, we look for a test with good statistical properties and that can be computed in polynomial time. Indeed, finding efficient statistical methods in  high-dimension is critical. Specifically, we study a test based on a natural convex (semidefinite) relaxation of $\lkm(\Sh)$ developed in \cite{dAsGhaJor07}.

For any $A \succeq 0$ let $\SDP_k(A)$ be defined as the optimal value of the following semidefinite program:
\begin{align}
\SDP_k(A)=  \ & \text{max.} && \tr (AZ)&\label{EQ:sdp}\\
& \text{subject to}&& \tr (Z) =1, \;|Z|_1 \leq k\,,\;  Z \succeq 0\nonumber& \, 
\end{align}
This optimization problem can be reformulated  as a semidefinite program in its canonical form with a polynomial number of constraints and can therefore be solved in polynomial time up to arbitrary precision using interior point methods for example \cite{BoyVan04}. Indeed, we can write
\begin{align*}
\SDP_k(A)=  \ & \text{max.} && \sum_{i,j} A_{ij}(z_{ij}^+ - z_{ij}^-)&\\
& \text{subject to} && z_{ij}^+=z_{ji}^+\ge 0,\; z_{ij}^-=z_{ji}^-\ge 0\nonumber \\
& && \sum_{i}(z_{ii}^+ - z_{ii}^-) =1, \;\sum_{i,j}(z_{ij}^+ + z_{ij}^-) \leq k&\nonumber\\
& && \sum_{i> j}(z_{ij}^+ - z_{ij}^-)(E_{ij}+E_{ji}) +\sum_{\ell}(z_{\ell \ell}^+ - z_{\ell \ell}^-)E_{\ell \ell}  \succeq 0\nonumber \, .
\end{align*}

Consider the following test
\begin{equation}
\label{EQ:defpsiSDP}
\psi_{d,n,k} = \mathbf{1}\{\SDP^{(n)}_k(\Sh) > 1+\tau\}\, ,\quad \tau>0\,,
\end{equation}
where  $\SDP^{(n)}_k$ is a $1/\sqrt{n}$-approximation of $\SDP_k$.  \cite{BacAhidAs10} show that $\SDP^{(n)}_k$ can be computed in $\mathcal{O}(k d^3\sqrt{n \log d})$ elementary operations and thus in polynomial time. 
\begin{theorem}
\label{TH:sdpub}
Assume that $(d,n,k)$ are such that
$$
\tilde\theta= 23 \sqrt{\frac{ k^2\log(4d^2/ \delta)}{n}} \le 1\,.
$$ 
Then, for $\theta \in [\tilde \theta, 1]$, the test $\psi$ defined in~\eqref{EQ:defpsiSDP} with threshold
$
\tau=16 \sqrt{\frac{k^2 \log(4d^2/\delta)}{n}}+\frac{1}{\sqrt n}\, ,
$
satisfies
$$
\sup_{\begin{subarray}{l}\fP_0 \in \cD_0\\\fP_1 \in \cD_1^k (\theta)\end{subarray}} \Big\{ \fP_0^{\otimes n}(\psi=1) \vee \fP_1^{\otimes n}(\psi=0) \Big \} \le \del \, .
$$
\end{theorem}
\begin{proof}
Define
\begin{eqnarray*}
\tau_0 =  16 \sqrt{\frac{ k^2\log(4d^2/ \delta)}{n}}\,, \qquad  \tau_1 = 7 \sqrt{\frac{k\log(4/\delta)}{n}}\, .
\end{eqnarray*}
For all $\delta >0$,  $\fP_0 \in \cD_0, \fP_1 \in \cD^k_1(\theta)$,  by  Lemma~\ref{LEM:SDPH1} and Lemma~\ref{LEM:lkmH0}, since $\SDP_k(\Sh) \ge \lkm(\Sh)$,  it holds
$$
\fP_0^{\otimes n}\Big(\SDP_k(\Sh) \ge 1+ \tau_0 \Big) \le \delta \, , \quad  
\fP_1^{\otimes n}\Big(\SDP_k(\Sh) \le 1+ \theta - \tau_1\Big) \le \delta\, .$$
Recall that   $|\SDP_k^{(n)}-\SDP_k|\le 1/\sqrt{n}$ and observe that $ \tau_0 +1/\sqrt{n} = \tau \le \tilde \theta -  \tau_1 \le \theta -\tau_1$. 
%
%
\end{proof}
This size of the detection threshold $\tilde \theta$ is consistent with the  results of  \cite{AmiWai08,BerRig12} for Gaussian distributions. 

Clearly, this theorem, together with Theorem~\ref{TH:infminimax}, indicate that the test based on $\SDP$ may be suboptimal within the class of all tests. However, as we will see in the next section, it can be proved to be optimal in a restricted class of computationally efficient tests.

\section{Complexity theoretic lower bounds}
\label{SEC:CTLB}


%
It is legitimate to wonder if the upper bound in Theorem~\ref{TH:sdpub} is tight. Can faster rates be achieved by this method, or by other, possibly randomized, polynomial time testing methods? Or instead, is this gap intrinsic to the problem? A partial answer to this question is provided in \cite{BerRig12}, where it is proved that the test defined in~\eqref{EQ:defpsiSDP} cannot discriminate at a level significantly lower than $\tilde \theta$. Indeed, such a test could otherwise be used to solve instances of the planted clique problem that are believed to be hard. This result is supported by some numerical evidence as well.  

In this section, we show that it is true not only of the test based on SDP but of \emph{any} test computable in randomized polynomial time.

\subsection{Lower bounds and polynomial time reductions}

The upper bound of Theorem~\ref{TH:sdpub}, if tight, seems to indicate that there is a gap between the detection levels that can be achieved by any test, and those that can be achieved by methods that run in polynomial time. In other words, it indicates a potential statistical cost for computational efficiency. To study this phenomenon, we take the approach favored in theoretical computer science, where our primary goal is to classify problems, rather than algorithms, according to their computational hardness. Indeed, this approach is better aligned with our definition of optimal rate of detection where lower bounds should hold for any tests. Unfortunately, it is difficult to derive a lower bound on the performance of \emph{any} candidate algorithm to solve a given problem. Rather, theoretical computer scientists have developed reductions from problem {\sf A} to problem {\sf B} with the following consequence: if problem {\sf B} can be solved in polynomial time, then so can problem {\sf A}. Therefore, if problem {\sf A} is believed to be hard then so is problem {\sf B}. Note that our reduction requires extra bits of randomness and is therefore a randomized polynomial time reduction.

 This question needs to be formulated from a statistical detection point of view. As mentioned above,  $\lkm$ can be proved to be {\sf NP}-hard to approximate. Nevertheless, such \emph{worst case} results are not sufficient to prove negative results on our \emph{average case} problem. Indeed, the matrix is $\Sh$ is random and we only need to be able to approximate $\lkm(\Sh)$ up to constant factor on most realizations. In some cases, this small nuance can make a huge difference, as problems can be hard in the worst case but easy in average (see, e.g., \cite{Bop87} for an illustration on {\sf Graph Bisection}).
In order to prove a complexity theoretic lower bound on the sparse principal component detection problem, we will build a reduction from a notoriously hard detection problem: the planted clique problem.

\subsection{The Planted Clique problem}

Fix an integer $m \ge 2$ and let $\G_m$ denote the set of undirected graphs on $m$ vertices.
Denote by $\cG(m,1/2)$ the distribution over $\G_m$ generated by choosing to connect every pair of vertices by an edge independently with probability $1/2$. For any $\kappa \in \{2,\ldots,m\}$, the distribution $\cG(m,1/2,\kappa)$ is constructed by picking $\kappa$ vertices arbitrarily and placing a clique\footnote{A clique is a subset of fully connected vertices.} between them, then connect every other pair of vertices by an edge independently with probability $1/2$. Note that $\cG(m,1/2)$ is simply the distribution of an Erd{\H{o}}s-R{\'e}nyi random graph. In the decision version of this problem, called  {\sf Planted Clique},  one is given a graph $G$ on $m$ vertices and the goal is to detect the presence of a planted clique. 
\begin{definition}
\label{DEFN:pc}
Fix $m \ge \kappa>2$. Let {\sf Planted Clique} denote the following statistical hypothesis testing problem:
\begin{eqnarray*}
H_0^{\sf PC} &:& G \sim \cG(m,1/2)=\fP_0^{(G)}\\
H_1^{\sf PC} &:& G \sim \cG(m,1/2,\kappa)=\fP_1^{(G)}\,.
\end{eqnarray*}
A test for the planted clique problem is a family $\xi = \{\xi_{m,\kappa}\}$, where $\xi_{m,\kappa}\,:\,\mathbb{G}_m \to \{0,1\}$. 
\end{definition}

The search version of this problem \cite{Jer92, Kuc95}, consists in finding the clique planted under $H_1^{\sf PC}$. The decision version that we consider here is traditionally attributed to Saks \cite{KriVu02, HazKra11}. It is known  \cite{Spe94}  that  if $\kappa > 2\log_2 (m)$, the planted clique is the only clique of size $\kappa$ in the graph, asymptotically almost surely (a.a.s.). Therefore, a test based on the largest clique of $G$ allows to distinguish $H_0^{\sf PC}$ and $H_1^{\sf PC}$ for $\kappa> 2\log_2(m)$, a.a.s. This is clearly not a computationally efficient test. 

For $\kappa=o(\sqrt{m})$ there is no known polynomial time algorithm that solves this problem. Polynomial time algorithms for the case $\kappa=C\sqrt{m}$ were first proposed in \cite{AloKriSud98}, and subsequently in \cite{McS01, AmeVav11, DekGurPer10,FeiRon10,FeiKra00}. It is widely believed that there is no polynomial time algorithm that solves {\sf Planted Clique}  for any $\kappa$ of order $m^c$ for some fixed positive $c<1/2$. Recent research has been focused on proving that certain algorithmic techniques, such as the Metropolis process \cite{Jer92} and the Lov\`asz-Schrijver hierarchy of relaxations \cite{FeiKra03} fail at this task. The confidence in the difficulty of this problem is so strong that it has led researchers to prove impossibility results assuming that {\sf Planted Clique} is indeed hard. Examples include cryptographic applications, in \cite{JuePei00}, testing for $k$-wise dependence in \cite{AloAndKau07}, approximating Nash equilibria in \cite{HazKra11} and approximating solutions to the densest $\kappa$-subgraph problem by~\cite{AloAroMan11}.

We therefore make the following assumption on the planted clique problem. 
Recall that $\delta$ is a confidence level fixed throughout the paper.

\begin{hypothesis}{${\sf A_{PC}}$}
\namedlabel{HYP:apc}{${\sf A_{PC}}$}
For any $a,b \in (0,1), a<b$ and all randomized polynomial time tests $\xi=\{\xi_{m,\kappa}\}$, there exists a positive constant $\Gamma$ that may depend on $\xi, a,b$ and such that
$$
\fP_0^{(G)}(\xi_{m,\kappa}(G)=1)\vee\fP_1^{(G)}(\xi_{m,\kappa}(G)=0)\ge 1.2\delta\,, \quad \forall \  m^{\frac{a}{2}}<\Gamma \kappa<  m^{\frac{b}{2}}\,.
$$
%
\end{hypothesis}
Note that $1.2\delta<1/2$ can be replaced by any constant arbitrary close to $1/2$. Since $\kappa$ is polynomial in $m$, here a \emph{randomized polynomial time} test is a test that can be computed in time at most polynomial in $m$ and has access to extra bits of randomness. The fact that $\Gamma$ may depend on $\xi$ is due to the asymptotic nature of polynomial time algorithms. Below is an equivalent formulation of Hypothesis~\ref{HYP:apc}.

\begin{hypothesis}{${\sf B_{PC}}$}
For any $a,b \in (0,1), a<b$ and all randomized polynomial time tests $\xi=\{\xi_{m,\kappa}\}$, there exists $m_0\ge 1$ that may depend on $\xi, a, b$ and such that
$$
\fP_0^{(G)}(\xi_{m,\kappa}(G)=1)\vee\fP_1^{(G)}(\xi_{m,\kappa}(G)=0)\ge 1.2\delta\,, \quad \forall \  m^\frac{a}{2}<\kappa< m^{\frac{b}{2}}\,,  \ m \ge m_0\,.
$$
\end{hypothesis}
Note that we do not specify a computational model intentionally. Indeed, for some restricted computational models, Hypothesis~\ref{HYP:apc} can be proved to be true for all $a<b\in (0,1)$ \cite{Ros10, FelGriRey13}. Moreover, for more powerful computational models such as Turing machines, this hypothesis is conjectured to be true. It was shown in~\cite{BerRig12} that improving the detection level of the test based on SDP would lead to a contradiction of Hypothesis~\ref{HYP:apc} for some $b \in (2/3, 1)$. Herefater, we extend this result to all randomized polynomial time algorithms, not only those based on SDP.

\subsection{Randomized polynomial time reduction}

Our main result is based on a randomized polynomial time reduction of an instance of the planted clique problem to an instance of the sparse PC detection problem.
In this section, we describe this reduction and call it the \emph{bottom-left transformation}. For any $\mu \in (0,1)$, define 
$$
R_\mu=R_0 \cap \{k \ge n^\mu\} \cap \{n<d\}\,.
$$
The condition $k \ge n^\mu$ is necessary since ``polynomial time" is an intrinsically asymptotic notion and for fixed $k$, computing $\lkm$  takes polynomial time in $n$. The condition $n<d$ is an artifact of our reduction and could potentially be improved. Nevertheless, it characterizes the high-dimensional setup we are interested in and allows us to shorten the presentation.

Given $(d,n,k) \in R_\mu$, fix integers $m, \kappa$ such that $n \le m<d$, $k\le \kappa\le m$ and let $G=(V,E) \in \G_{2m}$ be an instance of the planted clique problem with a potential clique of size $\kappa$. 
We begin by extracting a bipartite graph as follows. Choose $n$ right vertices $V_{{\tt right}}$ at random among the $2m$ possible and choose $m$ left vertices $V_{{\tt left}}$ among the $2m-n$ vertices that are not in $V_{{\tt right}}$. The edges of this bipartite graph\footnote{The ``bottom-left" terminology comes from the fact that  the adjacency matrix of this bipartite graph can be obtained as the bottom-left corner of the original adjacency matrix after a random permutation of the row/columns.} are $E\cap\{V_{{\tt left}}\times V_{{\tt right}}\}$. Next, since $d>m$, add $d-m \ge 1$ new left vertices and place an edge between each new left vertex and every old right vertex independently with probability $1/2$. Label the left (resp. right) vertices using a random permutation of $\{1, \ldots, d\}$ (resp. $\{1, \ldots, n\}$) and denote by $V'=(\{1, \ldots, d\}\times\{1, \ldots, n\}, E)$ the resulting $d\times n$ bipartite graph. Note that if $G$ has a planted clique of size $\kappa$, then  $V'$ has a planted biclique of random size. 

Let $B$ denote the $d \times n$ adjacency matrix of $V'$ and let $\eta_1, \ldots, \eta_n$ be $n$ i.i.d Rademacher random variables that are independent of all previous random variables. Define
$$
X_i^{(G)} = \eta_i (2B_i-1) \in \{-1,1\}^d\,,
$$
where $B_i$ denotes the $i$-th column of $B$. Put together, these steps define the bottom-left transformation $\bl\,:\,\G_{2m}\to  \R^{d\times n}$ of a graph $G$ by
\begin{equation}
\label{EQN:graphvec}
\bl(G)=\left( X_1^{(G)}, \ldots, X_n^{(G)} \right) \in \R^{d\times n}\,.
\end{equation}
Note that $\bl(G)$ can be constructed in randomized polynomial time in $d, n,k, \kappa, m$.

\subsection{Optimal detection over randomized polynomial time tests}

For any $\alpha \in [1,2]$, define the detection level $\theta_\alpha>0$ by $ \theta_\alpha=\sqrt{\frac{k^\alpha}{n}}\,.$

Up to logarithmic terms, it interpolates polynomially between the statistically optimal detection level $\theta^*$ and the detection level $\tilde \theta$ that is achievable by the polynomial time test based on $\SDP$. We have $\theta^*=\theta_1\sqrt{ \log d}$ and $\tilde \theta=C\theta_2\sqrt{ \log d}$ for some positive constant $C$.

\begin{theorem}
\label{TH:ctlowbnd}
Fix $\alpha \in [1,2), \mu \in (0,\frac{1}{4-\alpha})$ and define
\begin{equation}
\label{EQ:gamma}
a=2\mu\,, \qquad b=1-(2-\alpha)\mu\,.
\end{equation}
For any $\Gamma>0$, there exists a constant $L>0$ such that the following holds. For any $(d,n,k) \in R_\mu$,  there exists $m, \kappa$ such that $(2m)^\frac{a}2\le \Gamma \kappa \le (2m)^{\frac{b}{2}}$, a random transformation $\bl=\{\bl_{d,n, k, m, \kappa}\}$, $\bl_{d,n,k, m, \kappa}\,:\,\G_{2m}\to \R^{d\times n}$ that can be computed in polynomial time   and distributions $\fP_0 \in \cD_0, \fP_1 \in \cD_1^k(L\theta_\alpha)$ such that for any test $\psi=\{\psi_{d,n,k}\}$, we have
$$
\fP_0^{\otimes n}(\psi_{d,n,k}=1) \vee \fP_1^{\otimes n}(\psi_{d,n,k}=0)\ge \fP_0^{(G)}(\xi_{m,\kappa}(G)=1)\vee\fP_1^{(G)}(\xi_{m,\kappa}(G)=0)-\frac{\delta}{5}\,,
$$
where  $\xi_{m,\kappa}=\psi_{d,n,k}\circ \bl_{d,n,k, m, \kappa}$. 
\end{theorem}

\begin{proof}
Fix $(d,n,k)\in R_\mu, \alpha \in [1,2)$.
First, if $G$ is an Erd{\H{o}}s-R{\'e}nyi graph, $\bl(G)=\big(X_1^{(G)},\ldots,X_n^{(G)}\big)$ is an array of $n$ i.i.d. vectors of $d$ independent Rademacher random variables. Therefore $X_1^{(G)} \sim \fP_0^{\bl(G)} \in \cD_0$.

Second, if $G$ has a planted clique of size $\kappa$, let $\fP^{\bl(G)}$ denote the joint distribution of $\bl(G)$.
The choices of $\kappa$ and $m$ depend on the relative size of $k$ and $n$. Our proof relies on the following lemma.
\begin{lemma}
\label{LEM:totvar}
Fix $\beta>0$ and integers $m,\kappa, n, k$ such that $1\le n \le m$, $2\le k \le \kappa \le m$,
\begin{equation}
\label{EQ:condtotvar}
(a) \ \frac{m}{n}\ge \frac{8}{\beta\delta}\,, \qquad (b)\ \frac{n\kappa}{m}\ge 16\log\big(\frac{m}n\big)\,, \qquad (c)\   \frac{n\kappa}{m}\ge 8k\, .
\end{equation}
Moreover, define
$$
\bar \theta=\frac{(k-1)\kappa}{2m}\,,
$$
Let $G\sim \cG(2m, 1/2, \kappa)$ and $\bl(G) = \big(X_1^{(G)},\ldots,X_n^{(G)}\big) \in \R^{d\times n}$ be defined in \eqref{EQN:graphvec}. Denote by $\fP_{1}^{\bl(G)}$  the  distribution of $\bl(G)$. Then, there exists a distribution $\fP_1 \in \D_1^{k}(\bar\theta)$  such that
$$
\big\|\fP_1^{\bl(G)}-\fP_1^{\otimes n} \big\|_{\sf TV} \le \beta\delta\, .
$$
\end{lemma}
\begin{proof}
Let $S \subset \{1, \ldots, n\}$ (resp. $T \subset \{1, \ldots, d\}$) denote the (random) right  (resp. left) vertices of $V'$ that are in the planted biclique. 

Define the random variables
$$
\begin{array}{ll}
\varepsilon'_i=\mathbf{1}\{i \in S\}, &i=1, \ldots, n\\
\gamma'_j = \mathbf{1}\{j \in T \},& j=1, \ldots, d\, .
\end{array}
$$
On the one hand, if $i\notin S$, i.e., if $\eps'_i=0$, then  $X^{(G)}_i$ is a vector of independent Rademacher random variables. On the other hand, if $i\in S$,  i.e., if $\eps'_i=1$ then, for any $j=1, \ldots, d$, 
$$
X_{i,j}^{(G)}=Y'_{i,j}=\left\{
\begin{array}{ll}
\eta_i & \text{if}\ \gamma'_j =1\,,\\
r_{ij} & \text{otherwise,}
\end{array}\right.
$$
where $r=\{r_{ij}\}_{ij}$ is a $n \times d$ matrix of i.i.d Rademacher random variables.

We can therefore write
$$
X_i^{(G)}= (1-\varepsilon_i') r_{ i} + \varepsilon_i' Y'_{ i}\,,\quad i=1, \ldots, n\,,
$$
where $Y_i'=(Y_{i,1}', \ldots, Y_{i,d}')^\top$ and $r_i^\top$ is the $i$th row of $r$.

Note that the  $\eps_i'$s are not independent.  Indeed, they correspond to $n$ draws \emph{without} replacement from an urn that contains $2m$ balls (vertices) among which $\kappa$ are of type $1$ (in the planted clique) and the rest are of type $0$ (outside of the planted clique). Denote by $\bp_{\eps'}$ the joint distribution of $\eps'=(\eps_1', \ldots, \eps_n')$ and  define their ``with replacement" counterparts as follows. Let $\eps_1, \ldots, \eps_n$ be $n$ i.i.d. Bernoulli random variables with parameter $p=\frac{\kappa}{2m} \le \frac{1}{2}$. Denote by $\bp_{\eps}$ the joint distribution of $\eps=(\eps_1, \ldots, \eps_n)$.

We also replace the distribution of the $\gamma'_j$s as follows. Let $\gamma=(\gamma_1, \ldots, \gamma_n)$ have conditional distribution given $\eps$ be given by 
$$
\bp_{\gamma|\eps}(A)=\Pro\Big(\gamma' \in A \Big| \sum_{i=1}^d \gamma'\ge k, \eps'=\eps\Big)\,.
$$
Define $\big(X_1, \ldots, X_n\big)$ by
$$
X_i =  (1-\varepsilon_i) r_i + \varepsilon_i Y_i\,,\quad i=1, \ldots, n\,,
$$
where $Y_i \in \R^d$ has coordinates given by
$$
Y_{i,j}=\left\{
\begin{array}{ll}
\eta_i & \text{if}\ \gamma_j =1\\
r_{ij} & \text{otherwise}
\end{array}\right.
$$
With this construction, the $X_i$s are iid. Moreover, as we will see, the joint distribution $\fP_1^{\bl(G)}$ of $\bl(G)=\big(X_1^{(G)}, \ldots, X_n^{(G)}\big)$ is close in total variation to the joint distribution $\fP_1^{\otimes n}$ of $\big(X_1, \ldots, X_n\big)$. 

Note first that Markov's inequality yields
\begin{equation}
\label{EQ:cheb}
\Pro\Big(\sum_{i=1}^n \eps_i>\frac{\kappa}{2}\Big)\le \frac{2np}{\kappa} =  \frac{n}{m}\,.
\end{equation}
Moreover, given $\sum_{i=1}^n \eps_i=s$, we have $\sum_{i=1}^d \gamma_i \ge U\sim\cH(2m-n, \kappa-s, n)$. It follows from \cite{DiaFre80}, Theorem~$(4)$ that
$$
\Big\|\cH(2m-n, \kappa-s, n)-\cB\Big(n, \frac{\kappa-s}{2m-n}\Big)\Big\|_{\sf TV}\le \frac{4n}{2m-n}\le \frac{4n}{m}\,.
$$
Together with the Chernoff-Okamoto inequality~\cite{Dud99}, Equation (1.3.10), it yields 
$$
\Pro\Big(U<\frac{n(\kappa-s)}{2m-n}-\sqrt{\frac{n(\kappa-s)}{2m-n}\log\big(\frac{m}n\big)}\Big|\sum_{i=1}^n \eps_i=s\Big)\le \frac{n}{m}+\frac{4n}{m}=\frac{5n}{m}\,.
$$
Combined with~\eqref{EQ:cheb} and view of~\eqref{EQ:condtotvar}$(b, c)$, it implies that with probability $1-6n/m$, it holds
\begin{equation}
\label{EQ:controlT}
\sum_{j=1}^d \gamma_j\ge U\ge \frac{n \kappa}{4m}-\sqrt{\frac{n\kappa}{4m}\log\big(\frac{m}n\big)}\ge \frac{n \kappa}{8m} \ge k\,.
\end{equation}

Denote by $\bp$ the joint distribution of $(\eps_1, \ldots, \eps_n,\gamma_1,\ldots,\gamma_d)$ and by $\bp'$ that of $(\eps'_1, \ldots, \eps'_n,\gamma'_1,\ldots,\gamma'_d)$. Using again \cite{DiaFre80}, Theorem~$(4)$ and~\eqref{EQ:condtotvar}(a), we get
$$
\|\bp'-\bp  \|_{\sf TV}\le \frac{6n}{m}+  \|\bp_{\eps'}-\bp_{\eps}\|_{\sf TV} \le  \frac{6n}{m}+\frac{4n}{2m}=\frac{8n}{m}\le \beta\delta\, .
$$
Since  the conditional distribution of $\big(X_1, \ldots, X_n\big)$ given $(\eps, \gamma)$ is the same as that of $\bl(G)$ given $(\eps', \gamma')$, we have
$$
\|\fP_1^{\bl(G)}-\fP^{\otimes n}_1\|_{\sf TV} = \|\bp'-\bp\|_{\sf TV} \le \beta\delta\,.
$$
It remains to prove that $\bP_1 \in \cD_1^k(\bar \theta)$. Fix $\nu>0$ and define $Z \in \cB_0(k)$ by 
$$
Z_j=\left\{
\begin{array}{ll}
\gamma_j/\sqrt{k}\,, & \text{if} \ \sum_{i=1}^j \gamma_i \le k\\
0 & \text{otherwise.}
\end{array}\right.
$$
Denote by $S_Z\subset \{1, \ldots, d\}$, the support of $Z$. Next, observe that for any $x, \theta>0$,  it holds
\begin{equation}
\label{EQ:infave}
\inf_{v \in \cB_0(k)}\fP_1^{\otimes n}\Big(\hvar(v)-(1+\theta)<-x\Big) \le \fP_1^{\otimes n}\Big(\hvar(Z)-(1+\theta)<-x\Big)\, . 
\end{equation}
Moreover, for any $i=1, \ldots, n$
$$
(Z^\top X_i)^2 = \frac{1}{k}\Big(k\eps_i\eta_i+(1-\eps_i)\sum_{j \in S_Z}r_{ij}\Big)^2=\eps_ik + (1-\eps_i)\frac{1}{k}\Big(\sum_{j \in S_Z}r_{ij}\Big)^2\,.
$$
Therefore, since $Z$ is independent of the $r_{ij}$s,  the following equality holds in distribution:
$$
(Z^\top X_i)^2\ \substack{\mathrm{dist.}\\=} \ 1+\eps_i(k-1)+\frac{2(1-\eps_i)}{k}\sum_{\ell=1}^{{k \choose 2}}\omega_{i,\ell}\,,
$$
where $\omega_{i,\ell}, i,\ell \ge 1$ is a sequence of i.i.d Rademacher random variables that are independent of the $\eps_i$s. Note that by Hoeffding's inequality, it holds with probability at least $1-\nu/2$,
$$
\frac{2}{nk}\sum_{i=1}^n\sum_{\ell=1}^{{k \choose 2}}\omega_{i,\ell} \ge -\frac{4}{nk}\sqrt{2n{k \choose 2}\log(2/\nu)}\ge  -4\sqrt{\frac{\log(2/\nu)}{n}}\,.
$$
Moreover,  it follows from the Chernoff-Okamoto inequality \cite{Dud99}, Equation (1.3.10), that with probability at least $1-\nu/2$, it holds
$$
\frac{k-1}{n}\sum_{i=1}^n\eps_i \ge \frac{(k-1)}{n}np-\frac{k-1}{n}\sqrt{2np\log(2/\nu)}\,.
$$
Put together, the above two displays imply that with probability $1-\nu$, it holds
\begin{align*}
\hvar(Z)&>1+\frac{(k-1)\kappa}{2m}-\frac{k-1}{n}\sqrt{\frac{n\kappa}{m}\log(2/\nu)} -4\sqrt{\frac{\log(2/\nu)}{n}}\\
&\ge 1+\frac{(k-1)\kappa}{2m}-\sqrt{2k\frac{(k-1)\kappa}{2m}\frac{\log(2/\nu)}{n}} -4\sqrt{\frac{\log(2/\nu)}{n}}\\
&=1+\bar \theta-\sqrt{2k\bar \theta\frac{\log(2/\nu)}{n}} -4\sqrt{\frac{\log(2/\nu)}{n}}\, .
\end{align*}
Together with~\eqref{EQ:infave}, this completes the proof.
\end{proof}

Define $N=\lceil 40/\delta\rceil$.  Assume first that $k \ge M^{-1}n^{\frac{1}{4-\alpha}}$ where $M>0$ is a constant to be chosen large enough (see below). Take 
$
\kappa=\max\big(8,M\log(N)\big)Nk\,, m=Nn.
$ 
It implies that
$$
\bar \theta:=\frac{(k-1)\kappa}{2m}\ge \frac{Mk^2}{4n}\ge \frac{1}{4M^{1-\frac{\alpha}{2}}}\sqrt{\frac{k^\alpha}{n}}\, .
$$
Moreover, under these conditions, it is easy to check that \eqref{EQ:condtotvar} is satisfied with $\beta=1/5$ since   and we are therefore in a position to apply Lemma~\ref{LEM:totvar}. It implies that there exists $\fP_1 \in \cD_1^k(\bar \theta)$ such that
$
\big\|\fP_1^{\bl(G)}-\fP_1^{\otimes n} \big\|_{\sf TV} \le \delta/5\,.
$

Assume now that $k < M^{-1}n^{\frac{1}{4-\alpha}}$. 
Take $m, \kappa\ge 2$ to be the largest integers such that
$$
m \le 2N\big(nk^{2-\alpha}\big)^{\frac{1}{2-b}} \,  \quad \Gamma\kappa \le (2 m)^{\frac{b}{2}} \,.
$$
Note that $\Gamma\kappa\ge (2m)^{\frac{a}2}$. Let us now check condition~\eqref{EQ:condtotvar}. It holds, for $M$ large enough,
\begin{eqnarray*}
(a)&& \frac{m}{n}> \frac{N}{n} \big(n^{1+(2-\alpha)\mu}\big)^{\frac{1}{2-b}}=N\ge 40/\delta.\\
(b)&& \frac{n\kappa}{m}\ge \frac{1}{2\Gamma(4N)^\frac{b}{2}}\sqrt{\frac{n}{k^{2-\alpha}}} > \frac{ M^{1-\frac{\alpha}{2}}}{2\Gamma(4N)^\frac{b}{2}}n^{\frac{1}{4-\alpha}}\ge 16\log\Big(\frac{m}{n}\Big)\, .\\
(c)&&\frac{n\kappa}{m}\ge \frac{1}{2\Gamma(4N)^\frac{b}{2}}\sqrt{\frac{n}{k^{2-\alpha}}} > \frac{ M^{2-\frac{\alpha}{2}}}{2\Gamma(4N)^\frac{b}{2}}k\ge 8k \, .\\
\end{eqnarray*}
Under these conditions, \eqref{EQ:condtotvar} is satisfied with $\beta=1/5$   and we are therefore in a position to apply Lemma~\ref{LEM:totvar}. It implies that there exists $\fP_1 \in \cD_1^k(\bar \theta)$ such that
$
\big\|\fP_1^{\bl(G)}-\fP_1^{\otimes n} \big\|_{\sf TV} \le \delta/5\,,
$
where $
\bar \theta:=\frac{(k-1)\kappa}{2m}\ge \frac{1}{8\Gamma(4N)^\frac{b}{2}}\sqrt{\frac{k^{\alpha}}{n}}$, taking $
L=\min\Big(\frac{1}{4M^{\alpha-1}},\frac{1}{8\Gamma(4N)^\frac{b}{2}}\Big)\,,
$
yields that $\fP_1 \in \cD_1^k(L\theta_\alpha)$ for any $(d,n,k) \in R_\mu$. Moreover,  
%
%
%
%
$$
\fP^{(G)}_0(\psi\circ \bl(G)=1) \vee \fP^{(G)}_1(\psi\circ \bl(G)=0) \le  \fP_0^{\otimes n}(\psi=1) \vee \fP_1^{\otimes n}(\psi=0) +\delta/5\,.
$$
\end{proof}
Theorems~\ref{TH:sdpub} and~\ref{TH:ctlowbnd} imply the following result.
\begin{corollary}
Fix $\alpha \in [1,2), \mu \in (0,\frac{1}{4-\alpha})$. Conditionally on Hypothesis~\ref{HYP:apc}, the optimal rate of detection $\theta^\circ$ over the class of randomized polynomial time tests satisfies
$$
\sqrt{\frac{k^\alpha}{n}}\le \theta^\circ\le \sqrt{\frac{k^2\log d}{n}}\,, \quad (d,n,k) \in R_\mu\,.
$$
\end{corollary}
\begin{proof}
Let $\cT$ denote the class of randomized polynomial time tests. Since $\bl$ can be computed in randomized polynomial time, $\psi \in \cT$ implies that $\xi=\psi\circ\bl \in \cT$. Therefore, for all $(d,n,k) \in R_\mu$, 
\begin{equation*}
\inf_{\psi\in \cT}\fP_0^{\otimes n} (\psi=1) \vee \fP_1^{\otimes n}(\psi=0) \ge \inf_{\xi\in \cT}\fP_0^{(G)}(\xi(G)=1)\vee\fP_1^{(G)}(\xi(G)=0)-0.2\delta =\delta\,.
\end{equation*}
where the last inequality follows from Hypothesis~\ref{HYP:apc} with $a,b$ as in~\eqref{EQ:gamma}. Therefore $\theta^\circ \ge \theta_\alpha$. 
The upper bound follows  from Theorem~\ref{TH:sdpub}.
\end{proof}

The gap between $\theta^\circ$ and $\theta^*$ in Corollary~\ref{COR:minimax} indicates that the price to pay for using  randomized polynomial time tests for the sparse detection problem is essentially of order $\sqrt{k}$. 

\medskip

\noindent\textbf{Acknowledgments:} Philippe Rigollet is partially supported by the National Science Foundation grants DMS-0906424 and DMS-1053987. Quentin Berthet is partially supported by a Gordon S. Wu fellowship.

\bibliographystyle{amsalpha}
\bibliography{sparsedetectionqb2}

\newcommand{\etalchar}[1]{$^{#1}$}
\providecommand{\bysame}{\leavevmode\hbox to3em{\hrulefill}\thinspace}
\providecommand{\MR}{\relax\ifhmode\unskip\space\fi MR }
\providecommand{\MRhref}[2]{%
  \href{http://www.ams.org/mathscinet-getitem?mr=#1}{#2}
}
\providecommand{\href}[2]{#2}
\begin{thebibliography}{ABBDL10}

\bibitem[AAK{\etalchar{+}}07]{AloAndKau07}
Noga Alon, Alexandr Andoni, Tali Kaufman, Kevin Matulef, Ronitt Rubinfeld, and
  Ning Xie, \emph{Testing {$k$}-wise and almost {$k$}-wise independence},
  S{TOC}'07---{P}roceedings of the 39th {A}nnual {ACM} {S}ymposium on {T}heory
  of {C}omputing, ACM, New York, 2007, pp.~496--505. \MR{2402475 (2010a:68181)}

\bibitem[AAM{\etalchar{+}}11]{AloAroMan11}
Noga Alon, Sanjeev Arora, Rajsekar Manokaran, Dana Moshkovitz, and Omri
  Weinstein, \emph{On the inapproximability of the densest $\kappa$-subgraph
  problem}, Unpublished, April 2011.

\bibitem[ABBDL10]{AddBroDev09}
Louigi Addario-Berry, Nicolas Broutin, Luc Devroye, and G{\'a}bor Lugosi,
  \emph{On combinatorial testing problems}, Annals of Statistics \textbf{38}
  (2010), no.~5, 3063--3092.

\bibitem[ACBL12]{AriBubLug12}
Ery Arias-Castro, S{{\'e}}bastien Bubeck, and G{{\'a}}bor Lugosi,
  \emph{Detection of correlations}, Ann. Statist. \textbf{40} (2012), no.~1,
  412--435. \MR{3014312}

\bibitem[ACBL13]{AriBubLug13}
E.~Arias-Castro, S.~Bubeck, and G.~Lugosi, \emph{Detecting positive
  correlations in a multivariate sample}, Arxiv Preprint (2013),
  arXiv:1202.5536.

\bibitem[ACCP11]{AriCanPla11}
Ery Arias-Castro, Emmanuel~J. Cand{{\`e}}s, and Yaniv Plan, \emph{Global
  testing under sparse alternatives: {ANOVA}, multiple comparisons and the
  higher criticism}, Ann. Statist. \textbf{39} (2011), no.~5, 2533--2556.
  \MR{2906877}

\bibitem[ACV13]{AriVer13}
Ery Arias-Castro and Nicolas Verzelen, \emph{Community detection in random
  networks}, Arxiv Preprint (2013).

\bibitem[AKS98]{AloKriSud98}
Noga Alon, Michael Krivelevich, and Benny Sudakov, \emph{Finding a large hidden
  clique in a random graph}, Proceedings of the ninth annual ACM-SIAM symposium
  on Discrete algorithms (Philadelphia, PA, USA), SODA '98, Society for
  Industrial and Applied Mathematics, 1998, pp.~594--598.

\bibitem[AV11]{AmeVav11}
Brendan~P.W. Ames and Stephen~A. Vavasis, \emph{Nuclear norm minimization for
  the planted clique and biclique problems}, Mathematical Programming
  \textbf{129} (2011), 69--89 (English).

\bibitem[AW09]{AmiWai08}
Arash~A. Amini and Martin~J. Wainwright, \emph{High-dimensional analysis of
  semidefinite relaxations for sparse principal components}, Annals of
  Statistics \textbf{37} (2009), no.~5B, 2877--2921.

\bibitem[BAd10]{BacAhidAs10}
Francis Bach, Selin~Damla Ahipasaoglu, and Alexandre d'Aspremont, \emph{Convex
  relaxations for subset selection}, Arxiv Preprint (2010).

\bibitem[BI13]{ButIng13}
Cristina Butucea and Yuri~I. Ingster, \emph{Detection of a sparse submatrix of
  a high-dimensional noisy matrix}, Bernoulli (to appear) (2013).

\bibitem[BKR{\etalchar{+}}11]{BalKolRin11}
Sivaraman Balakrishnan, Mladen Kolar, Alessandro Rinaldo, Aarti Singh, and
  Larry Wasserman, \emph{Statistical and computational tradeoffs in
  biclustering}, NIPS 2011 Workshop on Computational Trade-offs in Statistical
  Learning (2011).

\bibitem[Bop87]{Bop87}
Ravi~B. Boppana, \emph{Eigenvalues and graph bisection: An average-case
  analysis}, Foundations of Computer Science, 1987., 28th Annual Symposium on,
  oct. 1987, pp.~280 --285.

\bibitem[BR12]{BerRig12}
Quentin Berthet and Philippe Rigollet, \emph{Optimal detection of sparse
  principal components in high dimension}, ArXiv:1202.5070 (2012).

\bibitem[BV04]{BoyVan04}
Stephen Boyd and Lieven Vandenberghe, \emph{Convex optimization}, Cambridge
  University Press, Cambridge, 2004. \MR{2061575 (2005d:90002)}

\bibitem[CJ13]{ChaJor13}
Venkat Chandrasekaran and Michael~I. Jordan, \emph{Computational and
  statistical tradeoffs via convex relaxation}, Proceedings of the National
  Academy of Sciences (2013).

\bibitem[CMW12]{CaiMa12}
T.~Tony Cai, Zongming Ma, and Yihong Wu, \emph{Sparse {PCA}: Optimal rates and
  adaptive estimation}, Arxiv Preprint (2012).

\bibitem[dBG12]{dAsBacGha12}
Alexandre d'Aspremont, Francis Bach, and Laurent~El Ghaoui, \emph{Approximation
  bounds for sparse principal component analysis}, ArXiv:1205.0121 (2012).

\bibitem[DF80]{DiaFre80}
Persi Diaconis and David Freedman, \emph{Finite exchangeable sequences}, Ann.
  Probab. \textbf{8} (1980), no.~4, 745--764. \MR{577313 (81m:60032)}

\bibitem[DGGP10]{DekGurPer10}
Yael Dekel, Ori Gurel-Gurevich, and Yuval Peres, \emph{Finding hidden cliques
  in linear time with high probability}, Arxiv Preprint (2010).

\bibitem[dGJL07]{dAsGhaJor07}
Alexandre d'Aspremont, Laurent~El Ghaoui, Michael~I. Jordan, and Gert R.~G.
  Lanckriet, \emph{A direct formulation for sparse {PCA} using semidefinite
  programming}, SIAM Review \textbf{49} (2007), no.~3, 434--448.

\bibitem[Dud99]{Dud99}
Richard Dudley, \emph{Uniform central limit theorems}, Cambridge University
  Press, 1999.

\bibitem[FGR{\etalchar{+}}13]{FelGriRey13}
Vitaly Feldman, Elena Grigorescu, Lev Reyzin, Santosh Vempala, and Ying Xiao,
  \emph{Statistical algorithms and a lower bound for planted clique},
  Proceedings of the {F}ourty-{F}ifth {A}nnual {ACM} {S}ymposium on {T}heory of
  {C}omputing, STOC 2013, 2013.

\bibitem[FK00]{FeiKra00}
Uriel Feige and Robert Krauthgamer, \emph{Finding and certifying a large hidden
  clique in a semirandom graph}, Random Structures Algorithms \textbf{16}
  (2000), no.~2, 195--208. \MR{1742351 (2001j:05109)}

\bibitem[FK03]{FeiKra03}
\bysame, \emph{The probable value of the {L}ov{\'a}sz-{S}chrijver relaxations
  for maximum independent set}, SIAM J. Comput. \textbf{32} (2003), no.~2,
  345--370 (electronic). \MR{1969394 (2004c:05137)}

\bibitem[FR10]{FeiRon10}
Uriel Feige and Dorit Ron, \emph{Finding hidden cliques in linear time}, 21st
  {I}nternational {M}eeting on {P}robabilistic, {C}ombinatorial, and
  {A}symptotic {M}ethods in the {A}nalysis of {A}lgorithms ({A}of{A}'10),
  Discrete Math. Theor. Comput. Sci. Proc., AM, Assoc. Discrete Math. Theor.
  Comput. Sci., Nancy, 2010, pp.~189--203. \MR{2735341 (2012b:05192)}

\bibitem[H{\aa}s96]{Has96}
Johan H{\aa}stad, \emph{Clique is hard to approximate within
  {$n^{1-\epsilon}$}}, 37th {A}nnual {S}ymposium on {F}oundations of {C}omputer
  {S}cience ({B}urlington, {VT}, 1996), IEEE Comput. Soc. Press, Los Alamitos,
  CA, 1996, pp.~627--636. \MR{1450661}

\bibitem[H{\aa}s99]{Has99}
\bysame, \emph{Clique is hard to approximate within {$n^{1-\epsilon}$}}, Acta
  Math. \textbf{182} (1999), no.~1, 105--142. \MR{1687331 (2000j:68062)}

\bibitem[HK11]{HazKra11}
Elad Hazan and Robert Krauthgamer, \emph{How hard is it to approximate the best
  nash equilibrium?}, SIAM J. Comput. \textbf{40} (2011), no.~1, 79--91.

\bibitem[Jer92]{Jer92}
Mark Jerrum, \emph{Large cliques elude the {M}etropolis process}, Random
  Structures Algorithms \textbf{3} (1992), no.~4, 347--359. \MR{1179827
  (94b:05171)}

\bibitem[JL09]{JohLu09}
Iain~M. Johnstone and Arthur~Yu Lu, \emph{On consistency and sparsity for
  principal components analysis in high dimensions}, J. Amer. Statist. Assoc.
  \textbf{104} (2009), no.~486, 682--693. \MR{2751448}

\bibitem[JP00]{JuePei00}
Ari Juels and Marcus Peinado, \emph{Hiding cliques for cryptographic security},
  Des. Codes Cryptogr. \textbf{20} (2000), no.~3, 269--280. \MR{1779310
  (2001k:94054)}

\bibitem[KBRS11]{KolBalRin11}
M.~Kolar, S.~Balakrishnan, A.~Rinaldo, and A.~Singh, \emph{Minimax localization
  of structural information in large noisy matrices}, Advances in Neural
  Information Processing Systems (2011).

\bibitem[Ku{\v{c}}95]{Kuc95}
Lud{\v{e}}k Ku{\v{c}}era, \emph{Expected complexity of graph partitioning
  problems}, Discrete Appl. Math. \textbf{57} (1995), no.~2-3, 193--212,
  Combinatorial optimization 1992 (CO92) (Oxford). \MR{1327775 (96c:68091)}

\bibitem[KV02]{KriVu02}
Michael Krivelevich and Van~H. Vu, \emph{Approximating the independence number
  and the chromatic number in expected polynomial time}, J. Comb. Optim.
  \textbf{6} (2002), no.~2, 143--155. \MR{1885488 (2002m:05181)}

\bibitem[Ma13]{Ma13}
Zongming Ma, \emph{Sparse principal component analysis and iterative
  thresholding}, Ann. Statist. (to appear) (2013).

\bibitem[Mas07]{Mas07}
Pascal Massart, \emph{Concentration inequalities and model selection}, Lecture
  Notes in Mathematics, vol. 1896, Springer, Berlin, 2007, Lectures from the
  33rd Summer School on Probability Theory held in Saint-Flour, July 6--23,
  2003, With a foreword by Jean Picard. \MR{MR2319879}

\bibitem[McS01]{McS01}
Frank McSherry, \emph{Spectral partitioning of random graphs}, 42nd {IEEE}
  {S}ymposium on {F}oundations of {C}omputer {S}cience ({L}as {V}egas, {NV},
  2001), IEEE Computer Soc., Los Alamitos, CA, 2001, pp.~529--537. \MR{1948742}

\bibitem[Ros10]{Ros10}
Benjamin Rossman, \emph{Average-{C}ase {C}omplexity of {D}etecting {C}liques},
  ProQuest LLC, Ann Arbor, MI, 2010, Thesis (Ph.D.)--Massachusetts Institute of
  Technology. \MR{2873600}

\bibitem[SN13]{SunNob13}
Xing Sun and Andrew~B. Nobel, \emph{{On the maximal size of large-average and
  ANOVA-fit submatrices in a Gaussian random matrix.}}, Bernoulli \textbf{19}
  (2013), no.~1, 275--294.

\bibitem[Spe94]{Spe94}
Joel Spencer, \emph{Ten lectures on the probabilistic method}, second ed.,
  CBMS-NSF Regional Conference Series in Applied Mathematics, vol.~64, Society
  for Industrial and Applied Mathematics (SIAM), Philadelphia, PA, 1994.
  \MR{1249485 (95c:05113)}

\bibitem[SSST12]{ShaShaTom12}
Shai Shalev-Shwartz, Ohad Shamir, and Eran Tomer, \emph{Using more data to
  speed-up training time}, Proceedings of the Fifteenth International
  Conference on Artificial Intelligence and Statistics April 21-23, 2012 La
  Palma, Canary Islands., JMLR W\&CP, vol.~22, 2012, pp.~1019--1027.

\bibitem[Tsy09]{Tsy09}
Alexandre~B. Tsybakov, \emph{Introduction to nonparametric estimation},
  Springer Series in Statistics, Springer, New York, 2009, Revised and extended
  from the 2004 French original, Translated by Vladimir Zaiats. \MR{2724359
  (2011g:62006)}

\bibitem[Ver10]{Vers10}
Roman Vershynin, \emph{Introduction to the non-asymptotic analysis of random
  matrices}, Arxiv Preprint (2010).

\bibitem[VL12]{VuLei12}
Vincent Vu and Jing Lei, \emph{Minimax rates of estimation for sparse pca in
  high dimensions}, Proceedings of the Fifteenth International Conference on
  Artificial Intelligence and Statistics April 21-23, 2012 La Palma, Canary
  Islands., JMLR W\&CP, vol.~22, 2012, pp.~1278--1286.

\bibitem[Zuc06]{Zuc06}
David Zuckerman, \emph{Linear degree extractors and the inapproximability of
  max clique and chromatic number}, Proceedings of the thirty-eighth annual ACM
  symposium on Theory of computing (New York, NY, USA), STOC '06, ACM, 2006,
  pp.~681--690.

\end{thebibliography}

\appendix

\section{Technical lemmas}

\begin{lemma}
\label{LEM:lkmH0}
For all $\fP_0 \in \cD_0$, and $t>0$, it holds
$$\fP_0\Big(\lkm(\Sh) > 1 + 4\sqrt{\frac tn} + t \frac tn \Big) \le \Big(\frac{ed}{k} \Big)^k 9^k e^{-t}\, .$$
\end{lemma}

\begin{proof}
We define the following events, for all $S \subset \{1,\ldots,d\}$, $u \in \R^p$, and $t>0$
\begin{eqnarray*}
\cA &=&\Big\{\lkm(\Sh) \ge 1 + 4 \sqrt{\frac tn} + 4 \frac tn \Big\} \\
\cA_S &=& \Big\{\lambda_{\max}(\Sh_S) \ge 1 + 4 \sqrt{\frac tn} + 4 \frac tn \Big\}\\
\cA_u &=& \Big\{u^\top \Sh u \ge 1 + 2 \sqrt{\frac tn} + 2 \frac tn \Big\}\, .
\end{eqnarray*}
By union on all sets of cardinal $k$, it holds 
$$\cA \subset \bigcup_{|S| =k}\cA_S\, .$$
Furthermore, let $\cN_S$, be a minimal covering $1/4$-net of $\cS^S$, the set of unit vectors with support included in $S$. It is a classical result that $|\cN_S| \le 9^k$ as shown in \cite{Vers10} and that it holds
$$\lambda_{\max}(\Sh_S-I_S) \leq 2\max_{u \in \cN_S} u^\top (\Sh-I_p) u\, .$$
Therefore it holds
$$\cA_S \subset \bigcup_{u \in \cN_S} \cA_u\, .$$
Hence, by union bound
$$\fP_0(\cA) \le \sum_{|S| =k}  \sum_{u \in \cN_S} \fP_0(\cA_u) \, .$$
By definition of $\cD_0$, $\fP_0(A_u) \le e^{-t}$ for $|u|_2=1$. The classical inequality ${d \choose k} \le \big(\frac{ed}{k} \big)^k$ yields the desired result. 
\end{proof}

\begin{lemma}
\label{LEM:SDPH1}
For all $\fP_0 \in \cD_0$, and $\del>0$, it holds
$$\fP_0\Big(\SDP_k(\Sh) \le 1+2 \sqrt{\frac{ k^2\log(4d^2/ \delta)}{n}}+ 2\frac{k \log(4d^2/ \delta)}{n}+ 2\sqrt{\frac{ \log(2d/\delta)}{n}} + 2\frac{ \log(2d/\delta)}{n}  \Big) \ge  1-\del\, .$$
\end{lemma}

\begin{proof}
We decompose $\Sh$ as the sum of its diagonal and off-diagonal matrices, respectively $\Dh$ and $\Ph$. Taking $U=-\Ph$ in the dual formulation of the semidefinite program \cite{BacAhidAs10, BerRig12} yields
\begin{equation}
\label{EQ:mindual}
\SDP_k(\Sh) =\min_{U \in \Sy^d} \big\{\lambda_{\max}(\Sh+U) + k |U|_{\infty} \big\}\le |\Dh|_{\infty} + k |\Ph|_{\infty}\, .
\end{equation}

We first control the largest off-diagonal element of $\Sh$ by bounding $|\hat{\Psi}|_{\infty}$ with high probability.
For every $i\neq j$, we have
\begin{eqnarray*}
\hat{\Psi}_{ij}&=&\frac{1}{2} \Big[ \frac 1n \sum_{\ell=1}^n [\frac{1}{2}(X_{\ell, i}+X_{\ell, j})^2 -1] - \frac 1n \sum_{\ell=1}^n [\frac{1}{2}(X_{\ell, i}-X_{\ell, j})^2 -1]\Big]\\
&=&\frac{1}{2} \Big[ \frac 1n \sum_{\ell=1}^n \Big[\Big(\frac{e_i^\top+e_j^\top}{\sqrt{2}}X_\ell\Big)^2 -1\Big] - \frac 1n \sum_{k=1}^n \Big[\Big(\frac{e_i^\top-e_j^\top}{\sqrt{2}}X_\ell\Big)^2 -1\Big]\Big]\, .
\end{eqnarray*}
By definition of $\cD_0$, it holds for $t>0$ that
\[
\fP_0 \Big(|\hat{\Psi}_{ij}| \geq 2 \,\sqrt{\frac{t}{n}}+ 2\frac{t}{n}\Big) \leq 4e^{-t} \, .
\]
Hence, by union bound on the off-diagonal terms, we get
\[
\fP_0 \Big( \max_{i<j}|\hat \Psi_{ij}| \geq 2 \,\sqrt{\frac{t}{n}}+ 2\frac{t}{n} \Big) \leq 2d^2e^{- t} \, .
\]
Taking $t= \log(4p^2/\delta)$ yields that under $\fP_0$ with probability $1-\delta/2$,  
\begin{equation}
\label{EQ:offdiag}
|\hat{\Psi}|_{\infty}\le 2 \sqrt{\frac{ \log(4d^2/ \delta)}{n}}+ 2\frac{ \log(4d^2/ \delta)}{n} \, .
\end{equation}

We control the largest diagonal element of $\Sh$ as follows.
We have by definition of $\hat{\Delta}$, for all~$i$
\[
\hat{\Delta}_{ii} = \frac{1}{n} \sum_{\ell=1}^n (e_i^\top X_{\ell})^2 \, .
\]
Similarly, by union bound over the $p$ diagonal terms, it holds
\begin{equation*}
\fP_0 \Big(|\Dh|_{\infty} \geq 1+2\sqrt{\frac{t}{n}} + 2\frac{t}{n}\Big) \leq d \, e^{-t} \, . 
\end{equation*}
Taking $t=  \log(2p/\delta) $ yields, under $\fP_0$ with probability $1-\delta/2$,
\begin{equation}
\label{EQ:diag}
|\Dh|_{\infty} \leq 1+ 2\sqrt{\frac{ \log(2d/\delta)}{n}} + 2\frac{ \log(2d/\delta)}{n} \, .
\end{equation} 
The desired result is obtained by plugging~\eqref{EQ:offdiag} and~\eqref{EQ:diag} into~\eqref{EQ:mindual}. 
\end{proof}

\end{document}